\documentclass{article}

\usepackage{fullpage}
\usepackage[utf8]{inputenc}
\usepackage{subfigure}
\usepackage[round]{natbib}
\usepackage{amsfonts,amssymb,amsmath,dsfont,amsthm}
\usepackage{enumitem}
\usepackage{xspace}
\usepackage{graphicx}

\newtheorem{theorem}{Theorem}
\newtheorem{lemma}{Lemma}

\newtheorem{definition}{Definition}

\newcommand{\bigO}{\mathcal{O}}
\newcommand{\exactbigO}{\Theta}
\newcommand{\smallo}{o}
\newcommand{\integers}{\mathds{Z}}
\newcommand{\reals}{\mathds{R}}

\newcommand{\mF}{\mathcal{F}}
\newcommand{\mH}{\mathcal{H}}

\newcommand*{\eg}{\textit{e.g.}\@\xspace}
\newcommand*{\ie}{\textit{i.e.}\@\xspace}
\newcommand*{\aka}{\textit{a.k.a.}\@\xspace}

\newcommand{\sg}{\operatorname{SG}}
\newcommand{\mg}{\operatorname{MG}}
\newcommand{\csg}{\operatorname{CSG}}
\newcommand{\core}{\operatorname{Core}}
\newcommand{\mcore}{\operatorname{MCore}}
\newcommand{\LD}{\operatorname{LD}}
\newcommand{\ld}{\operatorname{ld}}
\newcommand{\mgsimple}{\mg^{\setminus \LD}}
\newcommand{\sgpos}{\sg^{>0}}
\newcommand{\mgpos}{\mg^{>0}}
\newcommand{\IE}{\operatorname{IE}}
\newcommand{\MV}{\mathit{MV}}

\author{\'Elie de Panafieu\thanks{Email: depanafieuelie[at]gmail.com. This work was partially founded by 
Sorbonne Universit\'es, UPMC Univ Paris 06, CNRS, LIP6 UMR 7606, France;
the Austrian Science Fund (FWF) grant F5004;
the Amadeus program and the PEPS HYDrATA.}\\
{\normalsize Bell Labs France, Nokia}}

\title{Counting connected graphs with large excess\footnote{A shorter version of this work has been presented as a talk and published in the proceedings of the 28th International Conference on Formal Power Series and Algebraic Combinatorics (FPSAC 2016).}}

\begin{document}
\maketitle

\begin{abstract}
We enumerate the connected graphs 
that contain a linear number of edges with respect to the number of vertices.
So far, only the first term of the asymptotics was known.
Using analytic combinatorics,
\textit{i.e.}\ generating function manipulations,
we derive the complete asymptotic expansion.

\noindent \textbf{keywords.} connected graphs, analytic combinatorics, generating functions, asymptotic expansion
\end{abstract}

    \section{Introduction}

We investigate the number~$\csg_{n,k}$ of connected graphs 
with~$n$ vertices and~$n+k$ edges.
The quantity~$k$, defined as the difference 
between the numbers of edges and vertices, 
is the \emph{excess} of the graph.

    \paragraph{Related works}

Trees are the simplest connected graphs,
and reach the minimal excess~$-1$.
They were enumerated in 1860 by Borchardt,
and his result, known as \emph{Cayley's Formula},
is~$\csg_{n,-1} = n^{n-2}$.
\cite{R59} then derived the formula
for~$\csg_{n,0}$, which corresponds to connected graphs 
that contain exactly one cycle,
and are called \emph{unicycles}.
\cite{W80},
using generating function techniques, 
obtained the asymptotics of connected graphs 
for~$k=\smallo(n^{1/3})$.
This result was improved by \cite{FSS04},
who derived a complete asymptotic expansion
for fixed excess.

\cite{L90} obtained the asymptotics of~$\csg_{n,k}$
when~$k$ goes to infinity while~$k = \smallo(n)$.
\cite{BCM90} derived the asymptotics for a larger range,
requiring only that~$2k/n - \log(n)$ is bounded.
This covers the interesting case where~$k$ is proportional to~$n$.
Their proof was based on differential equations obtained by Wright,
involving the generating functions of connected graphs indexed by their excesses.
Since then, two simpler proofs were proposed.
The proof of \cite{PW05} relied on the enumeration 
of graphs with minimum degree at least~$2$.
The second proof, derived by \cite{HS06},
used probabilistic methods, 
analyzing a breadth-first search on a random graph.

\cite{ER60} proved that almost all graphs are connected
when~$(2k/n - \log(n))$ tends to infinity.
As a corollary, the asymptotics of connected graphs
with those parameters is equivalent to 
the total number of graphs.

    \paragraph{Contributions}

In this article, 
we derive an exact expression 
for the generating function of connected graphs (Theorem~\ref{th:exact_csg}),
tractable for asymptotics analysis.
Our main result is the following theorem.

\begin{theorem} \label{th:asymptotics_csg}
When~$k/n$ has a positive limit and~$d$ is fixed, then the following asymptotics holds
\[
    \csg_{n,k} = D_{n,k} \left( 1 + c_1 n^{-1} + \cdots + c_{d-1} n^{-(d-1)} + \bigO(n^{-d}) \right),
\]
where the dominant term~$D_{n,k}$ is derived in Lemma~\ref{th:dominant_asymptotics},
and the~$(c_{\ell})$ are computable constants.
\end{theorem}

    \section{Notations and models}

We introduce the notations adopted in this article,
the standard graph model, a multigraph model better suited for generating function manipulations,
and the concept of \emph{patchwork}, used to translate to graphs the results derived on multigraphs.

    \paragraph{Notations}

A \emph{multiset} is an unordered collection of objects,
where repetitions are allowed.
\emph{Sets}, or \emph{families}, are then multisets without repetitions.
A \emph{sequence}, or \emph{tuple}, is an ordered multiset.
We use the parenthesis notation~$(u_1, \ldots, u_n)$ for sequences, 
and the brace notation $\{u_1, \ldots, u_n\}$ for sets and multisets.
The cardinality of a set or multiset~$S$ is denoted by~$|S|$.
The double factorial notation for odd numbers stands for
\[
  (2k-1)!! = \frac{(2k)!}{2^k k!},
\]
and~$[z^n] F(z)$ denotes the~$n$th coefficient of the series expansion of~$F(z)$ at~$z=0$.

    \paragraph{Graphs}

We consider in this article the classic model of graphs, 
\aka \emph{simple graphs},
with labelled vertices and unlabelled unoriented edges.
All edges are distinct and no edge links a vertex to itself.
We naturally adopt for graphs generating functions
exponential with respect to the number of vertices,
and ordinary with respect to the number of edges
(see \cite{FS09}, or \cite{BLL97}).

\begin{definition}
A graph~$G$ is a pair~$(V(G), E(G))$,
where~$V(G)$ is the labelled set of vertices,
and $E(G)$ is the set of edges.
Each edge is a set of two vertices from~$V(G)$.
The number of vertices (resp.~of edges) is~$n(G) = |V(G)|$
(resp.~$m(G) = |E(G)|$).
The \emph{excess}~$k(G)$ is defined as~$m(G) - n(G)$.
The generating function of a family~$\mF$ of graphs is
\[
  F(z,w) =
  \sum_{G \in \mF}
  w^{m(G)}
  \frac{z^{n(G)}}{n(G)!},
\]
and~$F_k(z)$ denotes the generating function of multigraphs from~$\mF$ with excess~$k$,
\[
    F_k(z) = [y^k] F(z/y,y).
\]
\end{definition}

As always in analytic combinatorics and species theory,
the labels are distinct elements that belong to a totally ordered set.
When counting labelled objects (here, graphs), we always assume that 
the labels are consecutive integers starting at~$1$.
Another formulation is that we consider two objects as equivalent 
if there exists an increasing relabelling sending one to the other.

With those conventions, the generating function of all graphs is
\[
  \sg(z,w) = 
  \sum_{n \geq 0}
  (1+w)^{\binom{n}{2}}
  \frac{z^n}{n!},
\]
because a graph with~$n$ vertices
has~$\binom{n}{2}$ possible edges.
Since a graph is a set of connected graphs,
the generating function of connected graphs~$\csg(z,w)$
satisfies the relation
\[
  \sg(z,w) = e^{\csg(z,w)}.
\]
We obtain the classic closed form
for the generating function of connected graphs
\[
  \csg(z,w) =
  \log \bigg(
  \sum_{n \geq 0}
  (1+w)^{\binom{n}{2}}
  \frac{z^n}{n!}
  \bigg).
\]
This expression was the starting point of the analysis of \cite{FSS04},
who worked on graphs with fixed excess.
However, as already observed by those authors, it is complex to analyze,
because of ``\emph{magical}'' cancellations in the coefficients.
The reason of those cancellations is the presence of trees,
which are the only connected components with negative excess.
In this paper, we follow a different approach, closer to the one of \cite{PW05}:
we consider \emph{cores}, \ie graphs with minimum degree at least~$2$,
and add rooted trees to their vertices.
This setting produces all graphs without trees.

    \paragraph{Multigraphs}

As already observed by \cite{FKP89,JKLP93},
\emph{multigraphs} are better suited for generating function manipulations than graphs.
Exact and asymptotic results on connected multigraphs are available in \cite{ElieThesis}.
We propose a new definition for those objects, 
distinct but related with the one used by \cite{FKP89, JKLP93},
and link the generating functions of graphs 
and multigraphs in Lemma~\ref{th:multigraphs_to_graphs}.
We define a multigraph as a graph
with labelled vertices, and labelled oriented edges,
where loops and multiple edges are allowed.
Since vertices and edges are labelled, 
we choose exponential generating functions with respect to both quantities.
Furthermore, a weight~$1/2$ is assigned to each edge,
for a reason that will become clear in Lemma~\ref{th:multigraphs_to_graphs}.

\begin{definition}
A multigraph~$G$ is a pair~$(V(G), E(G))$,
where~$V(G)$ is the set of labelled vertices,
and $E(G)$ is the set of labelled edges 
(the edge labels are independent from the vertex labels).
Each edge is a triplet $(v,w,e)$,
where~$v$, $w$ are vertices,
and~$e$ is the label of the edge.
The number of vertices (resp.~number of edges, excess) is~$n(G) = |V(G)|$
(resp.~$m(G) = |E(G)|$, $k(G) = m(G) - n(G)$).
The generating function of the family~$\mF$ of multigraphs is
\[
  F(z,w) =
  \sum_{G \in \mF}
  \frac{w^{m(G)}}{2^{m(G)} m(G)!}
  \frac{z^{n(G)}}{n(G)!},
\]
and~$F_k(z)$ denotes the generating function~$[y^k] F(z/y,y)$.
\end{definition}

Figure~\ref{fig:patchwork} presents an example of multigraph.	
A major difference between graphs and multigraphs
is the possibility of loops and multiple edges.

\begin{definition}
A \emph{loop} (resp.\ \emph{double edge}) of a multigraph~$G$ is a subgraph~$(V, E)$ 
(\textit{i.e.}\ $V \subset V(G)$ and~$E \subset E(G)$)
isomorphic to the following left multigraph (resp.\ to one of the following right multigraphs).
\begin{center}
\includegraphics[scale=1.]{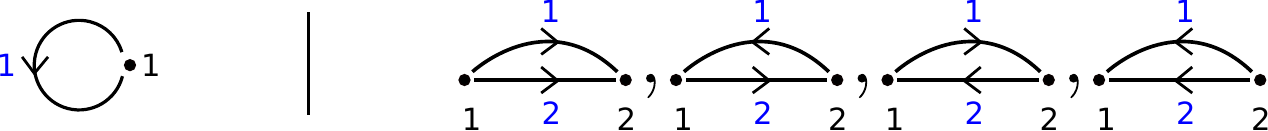}
\end{center}
The set of loops and double edges of a multigraph~$G$ is denoted by~$\LD(G)$,
and its cardinality by $\ld(G)$.
\end{definition}

In particular, a multigraph that has no double edge contains no multiple edge.
Multigraphs are better suited for generating function manipulations than graphs.
However, we aim at deriving results on the graph model,
since it has been adopted both by the graph theory and the combinatorics communities.
The following lemma, illustrated in Figure~\ref{fig:multigraphs_to_graphs}, 
links the generating functions of both models.

\begin{lemma} \label{th:multigraphs_to_graphs}
Let~$\mgsimple$ denote the family of multigraphs
that contain neither loops nor double edges,
and~$p$ the projection 
from~$\mgsimple$ to the set~$\sg$ of graphs,
that erases the edge labels and orientations,
as illustrated in Figure~\ref{fig:multigraphs_to_graphs}.
Let~$\mF$ denote a subfamily of~$\mgsimple$,
stable by edge relabelling and change of orientations.
Then there exists a family~$\mH$ of graphs such that $p^{-1}(\mH) = \mF$.
Furthermore, the generating functions of~$\mF$ and~$\mH$,
with the respective conventions of multigraphs and graphs,
are equal
\[
  \sum_{G \in \mF}
  \frac{w^{m(G)}}{2^{m(G)} m(G)!}
  \frac{z^{n(G)}}{n(G)!}
  =
  \sum_{G \in \mH}
  w^{m(G)}
  \frac{z^{n(G)}}{n(G)!}.
\]
\end{lemma}

\begin{figure}[htbp]
  \begin{center}
    \includegraphics[scale=0.9]{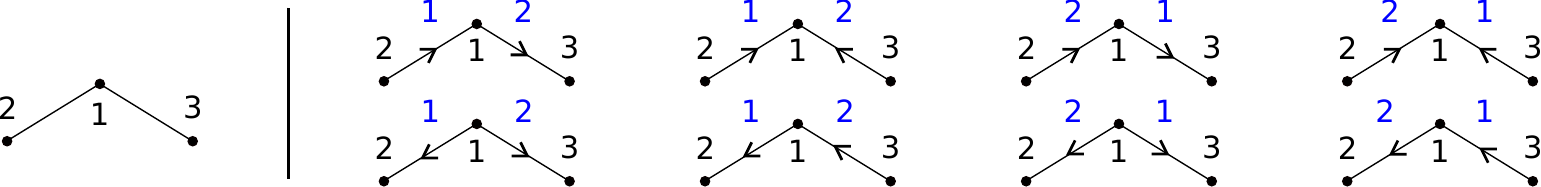}
    \caption{A graph~$G$ and the set~$\mF$ of multigraphs 
      sent by~$p$ (defined in Lemma~\ref{th:multigraphs_to_graphs}) to the graph~$G$.
      The generating function of~$\{G\}$ (resp.\ $\mF$) is $w^2 \frac{z^3}{3!}$ (resp.\ $8 \frac{w^2}{2^2 2!} \frac{z^3}{3!}$).
      As stated by Lemma~\ref{th:multigraphs_to_graphs}, those generating functions are equal.}
    \label{fig:multigraphs_to_graphs}
  \end{center}
\end{figure}

    \paragraph{Patchworks}

To apply the previous lemma, we need to remove
the loops and multiple edges from multigraph families.
Our tool is the inclusion-exclusion technique,
in conjunction with the notion of \emph{patchwork}.

\begin{definition} \label{def:patchworks}
A \emph{patchwork} with~$p$ parts
$
  P = \{(V_1, E_1), \ldots, (V_p, E_p)\}
$
is a set of~$p$ pairs $(\text{vertices}, \text{edges})$ such that
\[
  \mg(P) = \left( \cup_{i=1}^p V_i, \cup_{i=1}^p E_i \right)
\]
is a multigraph, and each~$(V_i, E_i)$ 
is either a loop or a double edge of~$\mg(P)$, \textit{i.e.}\ $P \subset \LD(\mg(P))$.
The number of parts of the patchwork is~$|P|$.
Its number of vertices~$n(P)$, edges~$m(P)$, and its excess~$k(P)$
are the corresponding numbers for~$\mg(P)$.
See Figure~\ref{fig:patchwork}.
\end{definition}

\begin{figure}[htbp]
  \begin{center}
    \includegraphics[scale=1]{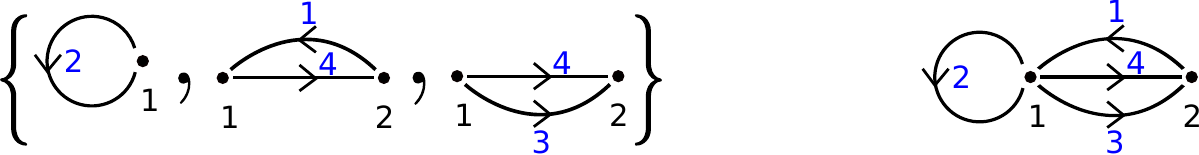}
    \caption{A patchwork~$P$ of excess~$2$, and the multigraph~$\mg(P)$. 
    Observe that several patchworks can lead to the same multigraph.
    Here, $\LD(\mg(P)) \neq P$,
    since the double edge~$(\{1,2\}, \{(2,1,1),(1,2,3)\})$ is missing from~$P$.}
    \label{fig:patchwork}
  \end{center}
\end{figure}

In particular, all pairs~$(V_i, E_i)$ are distinct,
$\mg(P)$ has minimum degree at least~$2$,
and two edges in~$E_i$, $E_j$ having the same label must link the same vertices.
We use for patchwork generating functions
the same conventions as for multigraphs
introducing an additional variable~$u$ to mark the number of parts
\[
    P(z,w,u) = \sum_{\text{patchwork~$P$}} u^{|P|} \frac{w^{m(P)}}{2^{m(P)} m(P)!} \frac{z^{n(P)}}{n(P)!}.
\]

\begin{lemma} \label{th:patchworks}
The generating function of patchworks is equal to
\[
    P(z,w,u) = \sum_{k \geq 0} P_k(zw,u) w^k, \qquad \text{where} \quad P_0(z,u) = e^{u \frac{z}{2} + u \frac{z^2}{4}}.
\]
For each~$k$, there is a polynomial~$P_k^{\star}(z,u)$ such that~$P_k(z,u) = P_0(z,u) P_k^{\star}(z,u)$.
\end{lemma}

\begin{proof}
A patchwork of excess~$0$ is a set of isolated loops and double edges
(\textit{i.e.}\ sharing no vertex with another loop or double edge),
which explains the expression of~$P_0(z,u)$.
Let~$P_k^{\star}$ denote the family of patchworks of excess~$k$
that contain no isolated loop or double edge.
Each vertex of degree~$2$ then belongs to exactly one double edge and no loop.
The number of such double edges is at most~$k$,
because each increases the excess by~$1$.
If we remove them, the corresponding multigraph has minimum degree at least~$3$ and excess at most~$k$.
There is a finite number of such multigraphs
(see \textit{e.g.}\ \cite{W80}, and we give the proof in Appendix~\ref{sec:cubic_multigraphs} for completeness),
so the family~$P_k^{\star}$ is finite,
and~$P_k^{\star}(z,u)$ is a polynomial.
Since any patchwork of excess~$k$ is a set of isolated loops 
and double edges and a patchwork from~$P_k^{\star}$, we have
\[
    P_k(z,u) = P_0(z,u) P_k^{\star}(z,u).
\]
\end{proof}

    \section{Exact enumeration}

In this section, we derive an exact expression for~$\csg_{k}(z)$, suitable for asymptotics analysis.
The proofs rely on tools developed by \cite{EdPR16, EdPCGGR16}.

\begin{theorem} \label{th:core}
The generating function of \emph{cores}, \ie graphs with minimum degree at least~$2$, is
\[
    \core(z,w) = \sum_{m \geq 0} (2m)! [x^{2m}] P(z e^x,w,-1) e^{z (e^x-1-x)} \frac{w^m}{2^m m!}.
\]
\end{theorem}

\begin{proof}
Let~$\mcore$ denote the set of multicores,
\textit{i.e.}\ multigraphs with minimum degree at least~$2$,
and set
\[
    \mcore(z,w,u) = \sum_{\text{multicore~$G$}} u^{\ld(G)} \frac{w^{m(G)}}{2^{m(G)} m(G)!} \frac{z^{n(G)}}{n(G)!},
\]
where~$\ld(G)$ denotes the number of loops and double edges in~$G$.
According to Lemma~\ref{th:multigraphs_to_graphs}, we have $\core(z,w) = \mcore(z,w,0)$.
To express the generating function of multicores, 
the inclusion-exclusion method (see \cite[Section III.7.4]{FS09})
advises us to consider~$\mcore(z,w,u+1)$ instead.
This is the generating function of the set~$\mcore^{\star}$ of multicores
where each loop and double edge is either marked by~$u$ or left unmarked.
The set of marked loops and double edges form, by definition, a patchwork.
One can cut each unmarked edge into two labelled half-edges.
Observe that the degree constraint implies that each vertex outside the patchwork contains at least two half-edges.
Reversely, as illustrated in Figure~\ref{fig:multicores}, 
any multicore from~$\mcore^{\star}$ can be uniquely build following the steps:
\begin{enumerate}
\item
start with a patchwork~$P$, which will be the final set of marked loops and double edges,
\item
add a set of isolated vertices,
\item
add to each vertex a set of labelled half-edges,
such that each isolated vertex receives at least two of them.
The total number of half-edges must be even, and is denoted by~$2m$,
\item
add to the patchwork the~$m$ edges obtained by
linking the half-edges with consecutive labels ($1$ with $2$, $3$ with $4$ and so on).
\end{enumerate}
\begin{figure}[htbp]
  \begin{center}
    \includegraphics[scale=0.9]{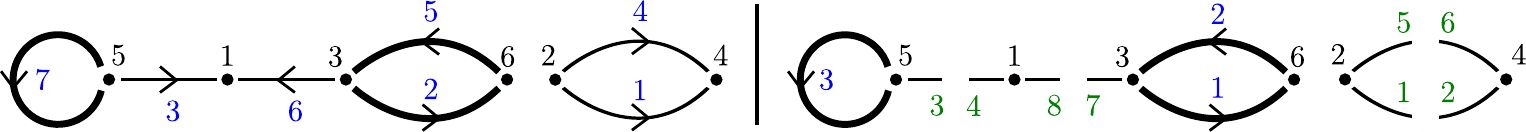}
    \caption{Left, a multigraph from~$\mcore^{\star}$ (the marked loops and double edges are bold).
        Right, the corresponding multigraph with labelled half-edges, build in step~$3$ of the proof of Theorem~\ref{th:core}.}
    \label{fig:multicores}
  \end{center}
\end{figure}
Observe that a relabelling of the vertices (resp.\ the edges) occurs at step~$2$ (resp.\ $4$).
This construction implies, by application of the species theory (\cite{BLL97}) or the symbolic method (\cite{FS09}),
the generating function relation
\[
    \mcore(z,w,u+1) = \sum_{m \geq 0} (2m)! [x^{2m}] P(z e^x,w,u) e^{z (e^x-1-x)} \frac{w^m}{2^m m!}.
\]
For~$u=-1$, we obtain the expression of~$\core(z,w) = \mcore(z,w,0)$.
\end{proof}

Any graph where no component is a tree can be built starting with a core,
and replacing each vertex with a rooted tree.
The components of smallest excess, zero, are then the unicycles.
The difference with the \emph{multi-unicycles} -- connected multigraphs of excess~$0$ --
is that the cycle can then be a loop or a double edge.
We recall the classic expressions of their generating functions (see \cite{FS09}).

\begin{lemma} \label{th:trees_and_unicycles}
The generating functions of rooted trees, multi-unicycles, and unicycles
are characterized by
\[
    T(z) = z e^{T(z)}, \qquad 
    \MV(z) = \frac{1}{2} \log \Big( \frac{1}{1-T(z)} \Big), \qquad 
    V(z) = \MV(z) - \frac{1}{2} T(z) - \frac{1}{4} T(z)^2.
\]
\end{lemma}

We apply the previous results to investigate graphs where all components have positive excess,
\ie that contain neither trees nor unicycles. 
This is the key new ingredient in our proof of Theorem~\ref{th:asymptotics_csg}.

\begin{lemma} \label{th:sgpos}
The generating function of graphs with excess~$k$ where each component has positive excess is
\[
    \sgpos_k(z) = 
    \sum_{\ell = 0}^k 
    (2(k-\ell)-1)!! 
    [x^{2(k-\ell)}]
    \frac{P_{\ell}(T(z) e^x,-1) e^{-V(z)}}
        {\big( 1 - T(z) \frac{e^x-1-x}{x^2/2} \big)^{k-\ell+1/2}}.
\]
It is coefficient-wise smaller than
\[
    \mgpos_k(z) =
    (2k-1)!! [x^{2k}] \frac{e^{-\MV(z)}}{\big( 1 - T(z) \frac{e^x-1-x}{x^2/2} \big)^{k+1/2}}.
\]
\end{lemma}

\begin{proof}
In the expression of the generating function of cores,
after developing the exponential as a sum over~$n$
and applying the change of variable~$m \leftarrow k+n$, we obtain
\[
    \core(z,w) =
    \sum_{k \geq 0}
    [x^{2k}]
    P(z e^x,w,-1) 
    \Bigg(
    \sum_{n \geq 0}
    \frac{(2(k+n))!}{2^{k+n} (k+n)!}
    \frac{\left( z w \frac{e^x-1-x}{x^2} \right)^n}{n!}
    \Bigg)
    w^k.
\]
The sum over~$n$ is replaced by its closed form
\[
    \core(z,w) =
    \sum_{k \geq 0}
    [x^{2k}]
    P(z e^x,w,-1)
    \frac{(2k-1)!!}
    {\left( 1 - z w \frac{e^x-1-x}{x^2/2} \right)^{k + 1/2}}
    w^k.
\]
Lemma~\ref{th:patchworks} is applied to expand~$P(z e^x,w,-1)$.
The generating function of cores of excess~$k$ is then
\[
    \core_k(z) = [y^k] \core(z/y,y) =
    \sum_{\ell = 0}^k 
    (2(k-\ell)-1)!! 
    [x^{2(k-\ell)}]
    \frac{P_{\ell}(z e^x,-1)}{\big( 1 - z \frac{e^x-1-x}{x^2/2} \big)^{k-\ell+1/2}}.
\]
If we do not remove the loops and double edges,
we obtain the generating function~$\mcore_k(z)$ of multicores of excess~$k$.
In the generating function, this means replacing~$P(z e^x,w,-1)$ with the constant~$1$,
so~$P_{\ell}$ vanishes except for~$\ell=0$, and
\[
    \mcore_k(z) = (2k-1)!! [x^{2k}] \frac{1}{\big( 1 - z \frac{e^x-1-x}{x^2/2} \big)^{k+1/2}}.
\]
A core of excess~$k$ where the vertices are replaced by rooted trees
can be uniquely decomposed as a set of unicycles, 
and a graph of excess~$k$ where each component has a positive excess, so
\[
    \core_k(T(z)) = e^{V(z)} \sgpos_k(z), 
    \qquad \mcore_k(T(z)) = e^{\MV(z)} \mgpos_k(z).
\]
This leads to the results stated in the lemma, after division by~$e^{V(z)}$ (resp.~$e^{\MV(z)}$).
According to Lemma~\ref{th:multigraphs_to_graphs}, 
the generating function $\mgpos(z,w)$ of multigraphs 
where all components have positive excess 
dominates coefficient-wise~$\sgpos(z,w)$,
so~$\mgpos_k(z) = [y^k] \mgpos(z/y,y)$ dominates coefficient-wise $\sg_k(z)$.
\end{proof}

Either by calculus -- as a corollary of the previous lemma -- or by a combinatorial argument,
we obtain the following result, first proven by Wright (see also \cite[Lemma~1 p.33]{JKLP93}), 
and that was a key ingredient of the proofs of \cite{BCM90, FSS04}.

\begin{lemma} \label{th:poly_sgpos}
For each~$k>0$, there exists a computable polynomial~$Q_k$ such that
\[
    \sgpos_k(z) = \frac{Q_k(T(z))}{(1-T(z))^{3k}}.
\]
\end{lemma}

Observe that this result is only useful for fixed~$k$.
We finally prove an exact expression for the number of connected graphs,
which asymptotics is derived in Section~\ref{sec:asymptotics}.

\begin{theorem} \label{th:exact_csg}
For~$k>0$, the number of connected graphs with~$n$ vertices and excess~$k$ is 
\[
    \csg_{n,k} = n! [z^n]\csg_k(z) = 
    \sum_{q=1}^k
    \frac{(-1)^{q+1}}{q}
    \sum_{\substack{k_1 + \cdots + k_q = k\\ \forall j,\ 1 \leq k_j \leq k-q+1}}
    n! [z^n] \prod_{j=1}^q \sgpos_{k_j}(z).
\]
\end{theorem}

\begin{proof}
Each graph in~$\sgpos$ is a set of connected graphs with positive excess, so
\[
    \sum_{\ell \geq 0} \sgpos_{\ell}(z) y^{\ell} = e^{\sum_{k > 0} \csg_k(z) y^k}.
\]
Observe that~$\sgpos_0(z) = 1$.
Indeed, the only graph of excess~$0$ where all components have positive excess is the empty graph
(this can also be deduced by calculus from Lemma~\ref{th:sgpos}).
Taking the logarithm of the previous expression and extracting the coefficient~$[y^k]$, we obtain
\[
    \csg_k(z) = [y^k] \log \bigg( 1 + \sum_{\ell \geq 1} \sgpos_{\ell}(z) y^{\ell} \bigg),
\]
which leads to the result by expansion of the logarithm and extraction of the coefficient~$[z^n]$.
Observe that~$q \leq k$ because each~$k_j$ is at least~$1$,
and~$k_j \leq k-q+1$ for the same reason.
\end{proof}

    \section{Asymptotics of connected graphs} \label{sec:asymptotics}

In this section, we prove Theorem~\ref{th:asymptotics_csg},
deriving~$\csg_{n,k}$ up to a multiplicative factor~$(1+\bigO(n^{-d}))$,
where~$d$ is an arbitrary fixed integer.
Our strategy is to express~$\csg_{n,k}$
as a sum of finitely many non-negligible terms, 
which asymptotic expansions are extracted using a saddle-point method.
We will see that in the expression of~$\csg_{n,k}$ from Theorem~\ref{th:exact_csg},
the dominant contribution comes from~$q=1$, \ie, applying Lemma~\ref{th:sgpos},
\[
    \csg_{n,k} \sim n! [z^n] \sgpos_k(z)
    =
    \sum_{\ell = 0}^{k}
    n!
    (2(k-\ell)-1)!!
    [z^n x^{2(k-\ell)}]
    \frac{P_{\ell}(T(z) e^x,-1) e^{-V(z)}}
        {\big( 1 - T(z) \frac{e^x-1-x}{x^2/2} \big)^{k-\ell+1/2}}.
\]
In this expression, the dominant contribution will come from~$\ell=0$.
This means that a graph with~$n$ vertices, excess~$k$, and without tree or unicycle components,
is connected with high probability -- a fact already proven by \cite{ER60} and used by \cite{PW05}.
Furthermore, its loops and double edges are typically disjoint, hence forming a patchwork of excess~$0$.
We now derive the asymptotics~$D_{n,k}$ of this dominant term,
and will use it as a reference, to which the other terms will be compared.

\begin{lemma} \label{th:dominant_asymptotics}
When~$k/n$ tends toward a positive constant, we have the following asymptotics
\[
    n!
    (2k-1)!!
    [z^n x^{2k}]
    \frac{P_{0}(T(z) e^x,-1) e^{-V(z)}}
        {\big( 1 - T(z) \frac{e^x-1-x}{x^2/2} \big)^{k+1/2}}
    \sim
    \frac{n^{n+k}}{\sqrt{2 \pi n}}
    \left( \frac{e^{\lambda/2}-e^{-\lambda/2}}{\lambda^{1+k/n}} \right)^n
    \frac{(e^{\lambda}-1-\lambda) e^{-\left(1+\frac{k}{2n}\right) \lambda}}
        {\sqrt{\frac{\lambda}{2} (e^{2\lambda} - 1 - 2 \lambda e^{\lambda})}},
\]
where the right-hand side is denoted by~$D_{n,k}$,
and~$\lambda$ is the unique positive solution of~$\frac{\lambda}{2} \frac{e^{\lambda}+1}{e^{\lambda}-1} = \frac{k}{n}+1$.
In particular, introducing the value~$\zeta$ characterized by~$T(\zeta) = \frac{\lambda}{e^{\lambda}-1}$, we have
\[
    D_{n,k} = \exactbigO \Bigg(
        \frac{1}{k}
        \frac{n! (2k-1)!!}
        {\big( 1 - T(\zeta) \frac{e^{\lambda}-1-\lambda}{\lambda^2/2} \big)^{k} \zeta^n \lambda^{2k}}
    \Bigg).
\]
\end{lemma}

\begin{proof}
Injecting the formulas for~$P_0(z,u)$ and~$V(z)$ derived in Lemmas~\ref{th:patchworks}, \ref{th:trees_and_unicycles},
the expression becomes
\[
    n! (2k-1)!! [z^n x^{2k}] A(z,x) B(z,x)^k,
\]
with
\[
    B(z,x) = \Big(1-T(z) \frac{e^x-1-x}{x^2/2}\Big)^{-1}
\]
and
\[
    A(z,x) = e^{-\frac{T(z) e^x}{2} - \frac{T(z)^2 e^{2x}}{4} + \frac{T(z)}{2} + \frac{T(z)^2}{4}} 
    \sqrt{(1-T(z)) B(z,x)}.
\]
We recognize the classic \emph{large powers} setting,
and a bivariate saddle-point method (see \eg \cite{BR99}) is applied to extract the asymptotics,
which implies the second result of the lemma:
\[
    n! (2k-1)!! [z^n x^{2k}] A(z,x) B(z,x)^k \sim
    n! (2k-1)!! \frac{A(\zeta,\lambda)}{2 \pi k \sqrt{\det(H(\zeta,\lambda))}} \frac{B(\zeta,\lambda)^k}{\zeta^n \lambda^{2k}}
\]
where~$\zeta$, $\lambda$ and the~$2 \times 2$ matrix~$(H_{i,j}(z,x))_{1 \leq i,j \leq 2}$ are characterized by the equations
\[
    \frac{\zeta \partial_{\zeta} B(\zeta,\lambda)}{B(\zeta,\lambda)} = \frac{n}{k},
    \quad
    \frac{\lambda \partial_{\lambda} B(\zeta,\lambda)}{B(\zeta,\lambda)} = 2,
    \quad
    H_{i,j}(e^{t_1}, e^{t_2}) = \partial_{t_i} \partial_{t_j} \log \left( B(e^{t_1}, e^{t_2}) \right).
\]
The first result follows by application of the Stirling formula and expansion of the expression.
The system of equation characterizing~$\zeta$ and~$\lambda$ is equivalent with
\[
    \frac{\lambda}{2} \frac{e^{\lambda} + 1}{e^{\lambda} - 1} = \frac{k}{n} + 1,
    \qquad
    T(\zeta) = \frac{\lambda}{e^{\lambda} - 1}.
\]
Since
\[
    n! (2k-!)!! \sim n^{n+k} \left( \frac{2k}{n} \right)^k e^{-n-k} \sqrt{2 \pi n} \sqrt{2},
\]
the super-exponential term in the asymptotics~$D_{n,k}$ is~$n^{n+k}$.
The exponential term is
\begin{equation} \label{eq:dnk1}
    \left( \frac{2k}{n} \right)^k e^{-n-k} \frac{B(\zeta, \lambda)^k}{\zeta^n \lambda^{2k}}
    = \left( \frac{e^{\lambda/2} - e^{-\lambda/2}}{\lambda^{1+k/n}} \right)^n.
\end{equation}
The coefficients of the symmetric matrix~$H = H(\zeta,\lambda)$ are
\[
    H_{1,1} = \frac{1}{1-T(\zeta)} \frac{n}{k} + \frac{n^2}{k^2},
    \quad
    H_{1,2} = H_{2,1} = \frac{2}{1-T(\zeta)} + \frac{2n}{k},
    \quad
    H_{2,2} = \lambda ( 1 - T(\zeta)) \frac{n}{k} + 2 \lambda.
\]
The constant and polynomial terms of the asymptotics~$D_{n,k}$ are
\begin{equation} \label{eq:dnk2}
    \sqrt{2 \pi n} \sqrt{2} \frac{A(\zeta, \lambda)}{2 \pi k \sqrt{\det(H(\zeta, \lambda))}}
    = \frac{1}{\sqrt{2 \pi n}}
    \frac{(e^{\lambda}-1-\lambda) e^{-(1 + \frac{k}{2n})\lambda}}{\sqrt{\frac{\lambda}{2} (e^{2\lambda}-1-2 \lambda e^{\lambda})}}.
\end{equation}
$D_{n,k}$ is then the product of~$n^{n+k}$ with the right-hand sides of Equations~\eqref{eq:dnk1} and~\eqref{eq:dnk2}.
\end{proof}

In the expression of~$\csg_{n,k}$ from Theorem~\ref{th:exact_csg},
the product over~$j$ has the following simple bound. 

\begin{lemma} \label{th:bound_prod_sgpos}
When~$k/n$ tends to a positive constant,
for any integer composition~$k_1 + \cdots + k_q = k$, we have
\[
    n! [z^n] \prod_{j=1}^q \sgpos_{k_j}(z) 
    = \frac{\prod_{j=1}^q (2k_j-1)!!}{(2k-1)!!} \bigO(k D_{n,k}),
\]
where the big~$\bigO$ is independent of~$q$.
\end{lemma}

\begin{proof}
According to Lemma~\ref{th:sgpos}, we have
\[
    n! [z^n] \prod_{j=1}^q \sgpos_{k_j}(z) \leq n! [z^n] \prod_{j=1}^q \mgpos_{k_j}(z)
    = \prod_{j=1}^q n! (2k_j-1)!! [z^n x^{2k_j}] \frac{e^{-\MV(z)}}{\big( 1 - T(z) \frac{e^x-1-x}{x^2/2} \big)^{k_j+1/2}}.
\]
Applying a classic bound (see \eg~\cite[Section~VIII.2]{FS09}), we obtain for all~$j$
\[
    [z^n x^{2k_j}] \frac{e^{-\MV(z)}}{\big( 1 - T(z) \frac{e^x-1-x}{x^2/2} \big)^{k_j+1/2}}
    \leq \frac{e^{-\MV(\zeta)}}
        {\big( 1 - T(\zeta) \frac{e^{\lambda}-1-\lambda}{\lambda^2/2} \big)^{k_j+1/2} \zeta^n \lambda^{2k_j}}.
\]
Taking the product over~$j$ and using the facts~$k_1 + \cdots + k_q = k$ and~$e^{-\MV(\zeta)} < 1$ leads to
\[
    n! [z^n] \prod_{j=1}^q \sgpos_{k_j}(z)
    \leq
    \frac{n! \prod_{j=1}^q (2k_j-1)!!}
        {\big( 1 - T(\zeta) \frac{e^{\lambda}-1-\lambda}{\lambda^2/2} \big)^{k+1/2} \zeta^n \lambda^{2k}}.
\]
The result follows, as a consequence of the bound derived in Lemma~\ref{th:dominant_asymptotics}.
\end{proof}

We now identify, in the expression of~$\csg_{n,k}$ from Theorem~\ref{th:exact_csg},
some negligible terms.

\begin{lemma} \label{th:remove_negligible_terms}
For any fixed~$d$ (resp.\ fixed~$d$ and~$q$), the following two terms are~$\bigO(k^{-d} D_{n,k})$
\[
    \sum_{q=d+5}^k \frac{(-1)^{q-1}}{q}
    \sum_{\substack{k_1 + \cdots + k_q = k\\ \forall j,\ 1 \leq k_j \leq k-q+1}}
    \prod_{j=1}^q
    n! [z^n] \sgpos_{k_j}(z),
    \qquad
    \sum_{j \geq 0}
    \sum_{\substack{k_1 + \cdots + k_q = k\\ \forall j,\ 1 \leq k_j \leq k-d-4-j}}
    \prod_{j=1}^q
    n! [z^n] \sgpos_{k_j}(z).
\]
\end{lemma}

\begin{proof}
According to Lemma~\ref{th:bound_prod_sgpos}, it is sufficient to prove that the sequence
\[
    S_{q,d,k} = \sum_{\substack{k_1 + \cdots + k_q = k\\ \forall j,\ 0 \leq k_j \leq k-d}}
    \frac{\prod_{j=1}^q (2k_j-1)!!}{(2k-1)!!}
\]
satisfies, for any fixed~$d$ (resp.\ when~$d$ and~$q$ are fixed),
\[
    \sum_{q = d+5}^k \frac{1}{q} S_{q,q-1,k} = \bigO(k^{-d-1})
    \qquad \text{and} \qquad
    \sum_{j \geq 0} S_{q,d+4+j,k} = \bigO(k^{-d-1}).
\]
The proof is available in Appendix~\ref{sec:double_factorial_asymptotics}
The two main ingredients are that the argument of the sum defining~$S_{q,d,k}$ is maximal
when one of the~$k_j$ is large (then the others remain small),
and that $S_{q,0,k} \leq 3 q$ for all~$q \leq k$ (proof by recurrence).
\end{proof}

Using the previous lemma, we remove the negligible terms from~$\csg_{n,k}$ and simplify its expression.

\begin{lemma} \label{th:bound_q_r}
There exist computable polynomials~$R_{q,r}$ such that, when~$k/n$ has a positive limit,
\begin{equation} \label{eq:asymptotics_csg_bounded_q_r}
    \csg_{n,k} =
    \sum_{q=1}^{d+4}
    (-1)^{q-1}
    \sum_{r=q-1}^{d+3}
    n! [z^n]
    \sgpos_{k-r}(z)
    \frac{R_{q,r}(T(z))}{(1-T(z))^{3r}}
    \big(1+\bigO(k^{-d}) \big).    
\end{equation}
\end{lemma}

\begin{proof}
The previous lemma proves that in the expression of~$\csg_{n,k}$ from Theorem~\ref{th:exact_csg},
we need only consider the terms corresponding to~$q \leq d+4$, and~$k-(d+4) \leq \max_j(k_j) \leq k$.
Since~$k_1+ \cdots + k_q = k$, when~$k$ is large enough and~$d$ is fixed,
there is at most one~$k_j$ between~$k-d$ and~$k$. 
Up to a symmetry of order~$q$, we can thus assume~$k_q = \max_j (k_j)$,
and introduce~$r = k - k_q$
\[
    \csg_{n,k} =
    \sum_{q=1}^{d+4}
    (-1)^{q-1}
    \sum_{r=q-1}^{d+3}
    n! [z^n]
    \sgpos_{k-r}(z)
    \sum_{\substack{k_1 + \cdots + k_{q-1} = r \\ \forall j,\ k_j \geq 1}}
    \prod_{j=1}^{q-1} \sgpos_{k_j}(z)
    \big(1+\bigO(k^{-d}) \big).
\]
According to Lemma~\ref{th:poly_sgpos}, there exist computable polynomials~$(Q_k)_{k \geq 1}$ such that
\[
    \sum_{\substack{k_1 + \cdots + k_{q-1} = r \\ \forall j,\ k_j \geq 1}}
    \prod_{j=1}^{q-1} \sgpos_{k_j}(z)
    =
    \sum_{\substack{k_1 + \cdots + k_{q-1} = r \\ \forall j,\ k_j \geq 1}}
    \prod_{j=1}^{q-1}
    \frac{Q_{k_j}(T(z))}{(1-T(z))^{3k_j}}
    =
    \frac{\sum_{\substack{k_1 + \cdots + k_{q-1} = r \\ \forall j,\ k_j \geq 1}}
    \prod_{j=1}^{q-1} Q_{k_j}(T(z))}{(1-T(z))^{3r}},
\]
and the numerator is the polynomial~$R_{q,r}$ evaluated at~$T(z)$.
\end{proof}

The next lemma proves that the terms corresponding to patchworks with a large excess are negligible.
The difficulty here is that we can only manipulate the generating functions of patchworks of finite excess.

\begin{lemma} \label{th:bound_ell}
When~$k/n$ has a positive limit and~$q$, $r$ are fixed, then
\[
    n! [z^n]
    \sgpos_{k-r}(z)
    \frac{R_{q,r}(T(z))}{(1-T(z))^{3r}}
\]
is equal to
\[
    \sum_{\ell = 0}^{d-1}
    n! (2(k-r-\ell)-1)!! 
    [z^n x^{2(k-r-\ell)}]
    \frac{P_{\ell}(T(z) e^x,-1) e^{-V(z)}}
        {\big( 1 - T(z) \frac{e^x-1-x}{x^2/2} \big)^{k-r-\ell+1/2}}
    \frac{R_{q,r}(T(z))}{(1-T(z))^{3r}}
    \big(1+\bigO(k^{-d})\big).
\]
\end{lemma}

\begin{proof}
We only present the proof of the equality
\[
    n! [z^n] \sgpos_{k}(z) = 
    \sum_{\ell = 0}^{d-1}
    n! (2(k-\ell)-1)!! 
    [z^n x^{2(k-\ell)}]
    \frac{P_{\ell}(T(z) e^x,-1) e^{-V(z)}}
        {\big( 1 - T(z) \frac{e^x-1-x}{x^2/2} \big)^{k-\ell+1/2}}
    \big(1+\bigO(k^{-d})\big).
\]
This corresponds to the case~$q=1$ and~$r=0$ of the lemma,
the general proof being identical.
Given a finite family~$\mF$ of multigraphs,
let~$\IE_{<d}(\mF)$ denote the bounded inclusion-exclusion operator
\[
    \IE_{<d}(\mF) =
    \sum_{G \in \mF\ }
    \sum_{P \subset \LD(G),\ k(P) < d}
    (-1)^{|P|}.
\]
Let~$\mgpos_{n,k}$ denote the set of multigraphs with~$n$ vertices, excess~$k$,
without tree or unicycle component.
Its subset~$\mgpos_{n,k,<d}$ (resp.\ $\mgpos_{n,k,\geq d}$)
corresponds to multigraphs~$G$ with maximal patchwork~$\LD(G)$ 
of excess less than~$d$ (resp.\ at least~$d$).
Given the decomposition
$
    \mgpos_{n,k} = \mgpos_{n,k,<d} \uplus \mgpos_{n,k,\geq d},
$
we have
\begin{equation} \label{eq:IE_linear}
    \IE_{<d}(\mgpos_{n,k}) = \IE_{<d}(\mgpos_{n,k,<d}) + \IE_{<d}(\mgpos_{n,k,\geq d}).
\end{equation}
Working as in the proof of Lemma~\ref{th:sgpos}, we obtain
\[
    \IE_{<d}(\mgpos_{n,k}) =
    \sum_{\ell = 0}^{d-1} 
    n! (2(k-\ell)-1)!! 
    [z^n x^{2(k-\ell)}]
    \frac{P_{\ell}(T(z) e^x,-1) e^{-V(z)}}
        {\big( 1 - T(z) \frac{e^x-1-x}{x^2/2} \big)^{k-\ell+1/2}}.
\]
Since~$(2(k-\ell)-1)!! = \exactbigO(k^{-\ell} (2k-1)!!)$, 
applying the same saddle-point method as in Lemma~\ref{th:dominant_asymptotics},
the~$\ell$th term of the sum is a~$\exactbigO(k^{-\ell} D_{n,k})$.
By inclusion-exclusion $\IE_{<d}(\mgpos_{n,k,<d}) = \sgpos_{n,k}$ so,
injecting those results in Equation~\eqref{eq:IE_linear},
\[
    \sgpos_{n,k} = 
    \sum_{\ell = 0}^{d-1} 
    n! (2(k-\ell)-1)!! 
    [z^n x^{2(k-\ell)}]
    \frac{P_{\ell}(T(z) e^x,-1) e^{-V(z)}}
        {\big( 1 - T(z) \frac{e^x-1-x}{x^2/2} \big)^{k-\ell+1/2}}
    - \IE_{<d}(\mgpos_{n,k,\geq d}).
\]
We now bound~$|\IE_{<d}(\mgpos_{n,k,\geq d})|$.
Any multigraph from~$\mgpos_{n,k,\geq d}$ contains, as a subgraph, a patchwork of excess~$d$.
Thus, $|\mgpos_{n,k,\geq d}|$ is bounded by the number of multigraphs from~$\mgpos_{n,k}$
where a patchwork of excess~$d$ is distinguished.
If, in any such multigraph, we mark another patchwork of excess less than~$d$
-- which might well intersect the patchwork previously distinguished -- we obtain the bound
\[
    |\IE_{<d}(\mgpos_{n,k,\geq d})|
    \leq
    \sum_{\ell=0}^{d-1}
    n! (2(k-d-\ell)-1)!! [z^n x^{2(k-d)}]
    \frac{P_{d}(T(z) e^x, 2) P_{\ell}(T(z) e^x,1)}
        {\big( 1 - T(z) \frac{e^x-1-x}{x^2/2} \big)^{k-d-\ell+1/2}},
\]
where the second argument of~$P_d$ is a $2$, because each loop and double edge 
of the distinguished patchwork can be either marked or left unmarked.
By the same saddle-point argument, this is a~$\bigO(k^{-d} D_{n,k})$.
\end{proof}

Combining Lemmas~\ref{th:bound_q_r} and~\ref{th:bound_ell},
$\csg_{n,k}$ is expressed as a sum of finitely many terms (since~$d$ is fixed)
\begin{equation} \label{eq:expansion_csg}
    \csg_{n,k} =
    \sum_{q=1}^{d+4}
    (-1)^{q-1}
    \sum_{r=q-1}^{d+3}
    \sum_{\ell=0}^{d-1}
    n! (2(k-r-\ell)-1)!! 
    [z^n x^{2k}]
    A_{q,r,\ell}(z,x) B(z,x)^k
    \big(1+\bigO(k^{-d})\big),
\end{equation}
where
\[
    B(z,x) = \Big(1-T(z) \frac{e^x-1-x}{x^2/2}\Big)^{-1}
\]
and
\[
    A_{q,r,\ell}(z,x) = \frac{x^{2(r+\ell)} P_{\ell}(T(z) e^x,-1) e^{-V(z)}}
        {\big(1-T(z) \frac{e^x-1-x}{x^2/2}\big)^{-r-\ell+1/2}} \frac{R_{q,r}(T(z))}{(1-T(z))^{3r}}.
\]
Since $(2(k-r-\ell)-1)!! = \exactbigO(k^{-r-\ell} (2k-1)!!)$, applying the same saddle-point method as in Lemma~\ref{th:dominant_asymptotics}, we obtain that the summand corresponding to~$q$, $r$, $\ell$ is a~$\exactbigO(k^{-r-\ell} D_{n,k})$.
Hence, $D_{n,k}$ is the dominant term in the asymptotics of~$\csg_{n,k}$.
We can be more precise in our estimation of each summand.
Its coefficient extraction is expressed as a Cauchy integral
on a torus of radii~$(\zeta, \lambda)$ (from Lemma~\ref{th:dominant_asymptotics}),
\[
    [z^n x^{2k}] A_{q,r,\ell}(z,x) B(z,x)^k =
    \frac{1}{(2\pi)^2}
    \int_{\theta=-\pi}^{\pi}
    \int_{\varphi=-\pi}^{\pi}
    A_{q,r,\ell}(\zeta e^{i \theta}, \lambda e^{i \varphi})
    \frac{B(\zeta e^{i \theta}, \lambda e^{i \varphi})^k}{\zeta^n e^{n i \theta} \lambda e^{2k i \varphi}}
    d\theta d\varphi,
\]
and its asymptotic expansion follows, by application of \cite[Theorem 5.1.2]{PW13}
\[
    n! (2(k-r-\ell)-1)!! 
    [z^n x^{2k}]
    A_{q,r,\ell}(z,x) B(z,x)^k
    =
    k^{-r-\ell} D_{n,k}
    \left( b_0 + \cdots + b_{d-1} n^{-d-1} + \bigO(n^{-d}) \right),
\]
where the~$(b_{\ell})$ are computable constants, and the factorials have been replaced by their asymptotic expansions.
Injecting those expansions in Equation~\eqref{eq:expansion_csg} concludes the proof of Theorem~\ref{th:asymptotics_csg}.

\bibliographystyle{abbrvnat}
\bibliography{/home/elie/research/articles/bibliography/biblio}

    \section{Appendix}

    \subsection{Multigraphs with minimum degree at least~$3$} \label{sec:cubic_multigraphs}

For completeness, we give the proof of the following Lemma, which goes back at least to~\cite{W80}.
It is applied in Lemma~\ref{th:patchworks}.

\begin{lemma}
The number of multigraphs with minimum degree at least~$3$ and excess~$k$ is finite.
\end{lemma}

\begin{proof}
Let us consider a multigraph~$G$ with minimum degree at least~$3$, $n$ vertices, $m$ edges, and excess~$k = m-n$.
Since the sum of the degrees is equal to twice the number of edges, we have
\[
    3 n \leq \sum_{v \in V(G)} \deg(v) = 2m,
\]
which implies
\[
    n \leq 2m - 2n = 2k
    \quad \text{and} \quad
    m = k + n \leq 3k.
\]
The number of multigraphs with at most~$2k$ vertices and~$3k$ edges is finite,
which concludes the proof.
\end{proof}

    \subsection{Properties of the sequence~$S_{q,d,k}$} \label{sec:double_factorial_asymptotics}

The goal of this section is to prove Lemma~\ref{th:Sqdk},
which is required for the proof of Lemma~\ref{th:remove_negligible_terms}.
We did not try to derive the tighter possible bounds.
Instead, their quality has been sacrificed in order to simplify the proofs.

The proofs are guided by the observation that
the argument of the sum defining~$S_{q,d,k}$ is maximal
when one of the~$k_j$'s is large (then the others remain small).
In my opinion, they are elementary, but too complicated compared to the simplicity of the result.
I am working on a more elegant version, starting with the integral representation
\[
    (2k-1)!! = \frac{1}{\sqrt{2\pi}} \int_{-\infty}^{+\infty} t^{2k} e^{-t^2/2} dt.
\]

\begin{lemma} \label{th:sum_q_2}
When~$d$ is fixed and~$k$ tends to infinity, we have
\[
    \sum_{r = d}^{k-d} \frac{(2(k-r)-1)!! (2r-1)!!}{(2k-1)!!}
    = \bigO(k^{-d}).
\]
\end{lemma}

\begin{proof}
Stirling's bounds
\[
    \sqrt{2 \pi} \sqrt{n} e^{-n} n^n \leq n! \leq e \sqrt{n} e^{-n} n^n
\]
imply
\[
    (2n-1)!! = \exactbigO(e^{-n} (2n)^n),
\]
and hence
\[
    \frac{(2(k-r)-1)!! (2r-1)!!}{(2k-1)!!}  = \exactbigO \left(\frac{(k-r)^{k-r} r^r}{k^k} \right).
\]
Using the symmetry, we cut the sum in the expression of the lemma in two halves
\[
    \sum_{r = d}^{k-d} \frac{(2(k-r)-1)!! (2r-1)!!}{(2k-1)!!}
    \leq
    2 \sum_{r = d}^{k/2} \frac{(2(k-r)-1)!! (2r-1)!!}{(2k-1)!!}.
\]
Injecting the previous relation, this implies
\[
    \sum_{r=d}^{k-d} \frac{(2(k-r)-1)!! (2r-1)!!}{(2k-1)!!} 
    = \bigO \Bigg( \sum_{r=d}^{k/2} \frac{(k-r)^{k-r} r^r}{k^k} \Bigg)
    = \bigO \Bigg( \sum_{r=d}^{k/2} \left( 1 - \frac{r}{k} \right)^{k-r} \left( \frac{r}{k} \right)^r \Bigg).
\]
Since $1-r/k \leq 1$ and $r/k \leq 1/2$, we have
\[
    \sum_{r=d}^{k/2} \left( 1 - \frac{r}{k} \right)^{k-r} \left( \frac{r}{k} \right)^r
    \leq \sum_{r=d}^{k/2} \left( \frac{r}{k} \right)^d \left( \frac{r}{k} \right)^{r-d}
    \leq k^{-d} \sum_{r \geq d} \frac{r^d}{2^{r-d}}
    = \bigO(k^{-d}).
\]
\end{proof}

\begin{lemma} \label{th:const_bound_S}
The sequence
\[
    S_{q,d,k} = 
    \sum_{\substack{k_1 + \cdots + k_q = k\\ \forall j,\ 0 \leq k_j \leq k-d}}
    \frac{\prod_{j=1}^q (2k_j-1)!!}{(2k-1)!!}
\]
satisfies $S_{q,d,k} \leq 3 q$ for all large enough~$k$, and any~$q \leq k$ and $d$.
\end{lemma}

\begin{proof}
Since~$S_{q,d,k} \leq S_{q,0,k}$, we focus on the case~$d=0$.
Up to a symmetry of order~$q$, we can assume~$k_q = \max_{j}(k_j)$
\[
    S_{q,0,k} \leq
    q
    \sum_{k_q = k/q}^{k}
    \frac{(2 k_q-1)!!}{(2k-1)!!}
    \sum_{\substack{k_1 + \cdots + k_{q-1} = k-k_q\\ \forall j,\ 0 \leq k_j \leq k_q}}
    \prod_{j=1}^q (2k_j-1)!!.
\]
We introduce~$r = k - k_q$, and replace the second sum with $(2r-1)!! S_{q-1,2r-k,r}$
\begin{equation} \label{eq:ineq_S}
    S_{q,0,k} \leq
    q
    \sum_{r = 0}^{k(1-1/q)}
    \frac{(2(k-r)-1)!! (2r-1)!!}{(2k-1)!!}
    S_{q-1,2r-k,r}.
\end{equation}
The biggest value reached by~$r$ is then at most~$k-1$.
Developping the first and last few terms, we obtain
\begin{equation} 
    S_{q,0,k} \leq
    q
    \Big( 
    S_{q-1,k,0} 
    + \frac{S_{q-1,k-1,1}}{2k-1}
    + \sum_{r = 2}^{k-2}
    \frac{(2(k-r)-1)!! (2r-1)!!}{(2k-1)!!}
    S_{q-1,2r-k,r}
    + \frac{S_{q-1,k-2,k-1}}{2k-1}
    \Big).
\end{equation}
We have
\[
    S_{q-1,k,0} \leq 1,
    \quad
    S_{q-1,k-1,1} \leq q,
    \quad
    S_{q-1,k-2,k-1} \leq \frac{2^k}{(2k-3)!!} = \smallo(1).
\]
We inject those relations in the previous inequality 
\begin{align*}
    S_{q,0,k} 
    &\leq
    q \left(
    1 + \frac{q}{2k-1} 
    + \sum_{r = 2}^{k-2}
    \frac{(2(k-r)-1)!! (2r-1)!!}{(2k-1)!!}
    S_{q-1,k-r,r}
    + \smallo(1)
    \right)
    \\ &\leq
    q \left( 2
    + \sum_{r = 2}^{k-2}
    \frac{(2(k-r)-1)!! (2r-1)!!}{(2k-1)!!}
    S_{q-1,2r-k,r}
    + \smallo(1)
    \right).
\end{align*}
Finally, we prove by recurrence~$S_{q,0,k} \leq 3q$ for~$k$ large enough.
For~$q=1$, we have~$S_{1,0,k} = 1$, which initializes the recurrence.
Let us assume that the recurrence holds for~$q-1$,
then~$S_{q-1,2r-k,r} \leq S_{q-1,0,r} \leq 3 q \leq 3 k$, and
\[
    S_{q,0,k} \leq q \left( 2 
    + \sum_{r = 2}^{k-2}
    \frac{(2(k-r)-1)!! (2r-1)!!}{(2k-1)!!}
    3 k
    + \smallo(1)
    \right).
\]
We apply Lemma~\ref{th:sum_q_2} to bound the sum
\[
    S_{q,0,k} \leq q(2 + \bigO(k^{-2}) 3 k + \smallo(1)),
\]
which is not greater than~$3q$ for~$k$ large enough.
\end{proof}

\begin{lemma} \label{th:exp_bound_S}
For any fixed~$d$, $k$ large enough and~$q \leq k$, we have~$S_{q,k-d,k} \leq 2^{-k}$.
\end{lemma}

\begin{proof}
The expression of~$S_{q,k-d,k}$ is
\[
    S_{q,k-d,k} =
    \sum_{\substack{k_1 + \cdots + k_q = k\\ \forall j,\ 0 \leq k_j \leq d}}
    \frac{\prod_{j=1}^q (2k_j-1)!!}{(2k-1)!!}.
\]
Applying Stirling's formula, we bound the double factorials
\[
    C_1 (2k)^k e^{-k} \leq (2k-1)!! \leq C_2 (2k)^k e^{-k}
\]
for some constant positive values~$C_1$, $C_2$. This implies, when~$k_1 + \cdots + k_q = k$,
\[
    \frac{\prod_{j=1}^q (2k_j-1)!!}{(2k-1)!!}
    \leq
    \frac{\prod_{j=1}^q C_2 (2k_j)^{k_j} e^{-k_j}}{C_1 (2k)^k e^{-k}}
    \leq
    \frac{C_2^q}{C_1}
    \left( \frac{d}{k} \right)^k
    \leq
    \frac{1}{C_1}
    \left( \frac{C_2 d}{k} \right)^k.
\]
The cardinality of the set~$\{(k_1, \ldots, k_q) \in [0, d]^q\ |\ k_1 + \cdots + k_q = k\}$
is at most~$d^q$, which is not greater than~$d^k$, so
\[
    S_{q,k-d,k} =
    \sum_{\substack{k_1 + \cdots + k_q = k\\ \forall j,\ 0 \leq k_j \leq d}}
    \frac{\prod_{j=1}^q (2k_j-1)!!}{(2k-1)!!}
    \leq
    \frac{1}{C_1}
    \left( \frac{C_2 d^2}{k} \right)^k.
\]
The right hand-side is smaller than~$2^{-k}$ when~$d$ is fixed and~$k$ is large enough.
\end{proof}

\begin{lemma} \label{th:bound_S}
For any fixed~$d$, all~$k$ large enough, and~$q \leq k$, we have~$S_{q,d,k} = \bigO(k^{-d+2})$ uniformly with respect to~$q$.
\end{lemma}

\begin{proof}
We start as in the proof of Lemma~\ref{th:const_bound_S}.
Up to a symmetry of order~$q$, we can assume~$k_q = max_{j}(k_j)$,
and introduce~$r = k-k_q$. We obtain an inequality similar to~\eqref{eq:ineq_S}
\[
    S_{q,d,k} \leq
    q
    \sum_{r = d}^{k(1-1/q)}
    \frac{(2(k-r)-1)!! (2r-1)!!}{(2k-1)!!}
    S_{q-1,2r-k,r}.
\]
We bound~$q$ by~$k$ and cut the sum into two parts
\[
    S_{q,d,k} \leq
    k
    \sum_{r = d}^{k-d}
    \frac{(2(k-r)-1)!! (2r-1)!!}{(2k-1)!!}
    S_{q-1,2r-k,r}
    +
    k
    \sum_{r = k-d+1}^{k-1}
    \frac{(2(k-r)-1)!! (2r-1)!!}{(2k-1)!!}
    S_{q-1,2r-k,r}.
\]
The first sum is bounded by application of Lemmas~\ref{th:const_bound_S} and~\ref{th:sum_q_2}
\[
    k
    \sum_{r = d}^{k-d}
    \frac{(2(k-r)-1)!! (2r-1)!!}{(2k-1)!!}
    S_{q-1,2r-k,r}
    \leq
    k
    \sum_{r = d}^{k-d}
    \frac{(2(k-r)-1)!! (2r-1)!!}{(2k-1)!!}
    3 k
    =
    \bigO(k^{-d+2}).
\]
In the second sum, the fraction of factorials is bounded by~$1$
\[
    k
    \sum_{r = k-d+1}^{k-1}
    \frac{(2(k-r)-1)!! (2r-1)!!}{(2k-1)!!}
    S_{q-1,2r-k,r}
    \leq
    k
    \sum_{r = k-d+1}^{k-1}
    S_{q-1,2r-k,r}.
\]
The sequence~$S_{q,d,k}$ is decreasing with respect to~$d$,
so when~$r$ is greater than~$k-d$, we have
\[
    S_{q-1,2r-k,r} = S_{q-1,r-(k-r),r} \leq S_{q-1, r-(d-1), r},
\]
which is bounded by~$2^{-r}$, according to Lemma~\ref{th:exp_bound_S}.
This implies
\[
    k
    \sum_{r = k-d+1}^{k-1}
    S_{q-1,2r-k,r}
    \leq
    k \sum_{r = k-d+1}^{k-1} 2^{-r}
    \leq \frac{k}{2^{k-d}},
\]
which is negligible compared to~$k^{-d+2}$.
\end{proof}

\begin{lemma} \label{th:Sqdk}
For any fixed~$d$ (resp.\ when~$d$ and~$q$ are fixed), we have
\[
    \sum_{q = d+5}^k \frac{1}{q} S_{q,q-1,k} = \bigO(k^{-d-1})
    \qquad \text{and} \qquad
    \sum_{j \geq 0} S_{q,d+4+j,k} = \bigO(k^{-d-1}).
\]
\end{lemma}

\begin{proof}
Since the sequence~$S_{q,d,k}$ is deacreasing with respect to~$d$, we have, when~$q$ is at least~$d+5$,
\[
    S_{q,q-1,k} \leq S_{q, d+4, k},
\]
which is a~$\bigO(k^{-d-2})$ according to Lemma~\ref{th:bound_S}. Hence,
\[
    \sum_{q = d+5}^k \frac{1}{q} S_{q,q-1,k} \leq k S_{q,d+4,k} = \bigO(k^{-d-1}).
\]
In the second sum of the lemma, the summand vanishes when~$j \geq k$, so
\[
    \sum_{j \geq 0} S_{q,d+4+j,k} \leq k S_{q,d+4,k} = \bigO(k^{-d-1}).
\]
\end{proof}

\end{document}